\documentclass{amsart}    
\usepackage[utf8]{inputenc}
\usepackage{graphicx,xcolor,cancel}
\usepackage{listings}
\usepackage[top=1in, bottom=1.50in, left=1.25in, right=1.25in]{geometry}

\usepackage{amsmath,amsfonts, amssymb,amsthm}
\usepackage{braket}

\theoremstyle{plain}
\newtheorem{thm}{Theorem}[section]
\newtheorem{lem}[thm]{Lemma}

\newtheorem{exm}{Example}
\theoremstyle{definition}

\begin{document}

\title{Quantization of two- and three-player cooperative games based on QRA}

\author{Ivan Eryganov, Jaroslav Hrdina and Ale\v s N\'avrat}

\address{ Institute of Mathematics, Faculty of Mechanical Engineering,  Brno University of Technology, Technick\' a 2896/2, 616 69 Brno, Czech Republic}

\email{Ivan.Eryganov@vut.cz, \{hrdina,a.navrat\}@fme.vutbr.cz}	

\begin{abstract}
In this paper, a novel quantization scheme for cooperative games is proposed. The considered circuit is inspired by the Eisert-Wilkens-Lewenstein protocol modified to represent cooperation between players and extended to $3$--qubit states. The framework of Clifford algebra is used to perform necessary computations. In particular, we use a direct analogy between Dirac formalism and Quantum Register Algebra to represent circuits. This analogy enables us to perform automated proofs of the circuit equivalence in a simple fashion. To distribute players' payoffs after the measurement, the expected value of the Shapley value with respect to quantum probabilities is employed. We study how entanglement, representing the level of pre-agreement between players, affects the final distribution of utility. The paper also demonstrates how all necessary calculations can be automatized using the Quantum Register Algebra and GAALOP software. 
\end{abstract}


\maketitle

\section{Introduction}
Game theory is a branch of applied mathematics that studies optimal decisions of entities between which conflict of interests has occurred \cite{Neumann}. This is done via formalization of the conflict into a game, where these entities take the role of players with clearly defined strategies (set of possible decisions) and payoff functions, describing utilities that can be obtained by the players \cite{Owen}. Cooperative game theory enables to study situations, where players are able to form coalitions in order to cooperate and improve their utilities \cite{Peleg}. In classical (superadditive) cooperative games, the main interest dwells in finding a fair allocation of the utility, produced by the grand coalition, with respect to players' contributions \cite{Saad}. 

Recently, game theory has been enriched with a new class of quantum games \cite{Flitney}. Originally, quantum game theory studied the quantization of non-cooperative matrix-form games, such as the well-known Prisoner's Dilemma \cite{Poundstone}. One of the most pioneering works on this topic \cite{Eisert} has studied new non-trivial equilibrium, which occurs due to the quantization of strategies and full entanglement between players' decision states. Though it has been an object of criticism \cite{Ben} (the considered strategy sets have no physical interpretation), this work has established the fundamental basis of contemporary research in quantum games \cite{Eryganov,bosc}. General quantum games can be characterized as games that include superposition of players' strategies and/or entanglement of players' decisions \cite{Flitney}.

In Eisert-Wilkens-Lewenstein (EWL) quantization protocol \cite{Eisert}, the decisions of the player to cooperate or deflect cooperation have been assigned to basis states of the qubit (quantum analogy of the classical bit \cite{Lima}). In this work, we extend this idea to describe the cooperative game of $n$-players by assigning basis states of the $n$--qubit to coalitions, where $1$ on $i$-th position means, that $i$-th player participates in this coalition, while $0$ on the same position identifies, that player is absent in the coalition. Whereas the proposed idea has the potential to be generalized into the domain of all cooperative games,  in this work, we study quantization protocols only for two- and three-player games. 

The main interest is to assess the potential of our quantization approach to cooperative games and to find out if it will allow for non-trivial results due to quantum phenomena. All the necessary calculations will be performed using geometric algebra \cite{dl,dhsa,hit,per,hil1,Lounesto}. The recent papers \cite{Cafaro,Alves} demonstrate an increasing number of applications of geometric algebras to quantum computing. In particular, we represent all the underlying quantum circuits using the language of Quantum Register Algebra (QRA) \cite{Hrdina2023}, a real form of complex Clifford algebra, \cite{hrdina2022quantum,Eryganov}. It should be emphasized, that complex Clifford algebra can be seen as a viable and intuitive symbolic alternative to the well-known fermionic quantum computation model \cite{Bravyi,Tilly}.
This enables us to study the properties of the proposed quantization schemes and to perform the automated proof of their equivalence. To validate the proposed approach, we "solve" instances of two- and three-player weighted majority cooperative games. The QRA representations of the corresponding $2$-- and $3$--qubit decision states and quantum gates on them will be programmed in GAALOP \cite{Alves,hil3}, which makes it possible to directly obtain the final state of the game, from which the probabilities will be extracted.

The rest of the paper is structured as follows. In Section 2, we briefly discuss cooperative game theory and its solution concept of the Shapley value \cite{Shapley}. Then, Section 3 is devoted to the introduction of the original quantization protocol for 2- and 3-player cooperative games. After that, in Section 4, we demonstrate how QRA can be applied to conveniently represent quantum computing and study the properties of quantum circuits. Section 5 presents a detailed description of the QRA representations of the considered quantum games. In the end, we demonstrate how quantum cooperative games can be solved using GAALOP and discuss the obtained results.  

\section{Cooperative games and Shapley value}
In this section, we will define cooperative games, their subclass of weighted majority games, and the underlying solution concept of the Shapley value according to \cite{Peleg}. Throuhout the paper, we focus solely on cooperative TU-games, where the value function assigns to coalitions a real number with a monetary equivalent, that can be freely distributed in any feasible way. More precisely, a general cooperative TU-game is conventionally defined by a pair $(N,v)$, where $N$ is a set of players and $v:2^N\rightarrow \mathbb{R}$ is the value function, which assigns to each non-empty subset $S\subseteq N,\ S\not=\emptyset$, called coalition, its utility represented by a real number $v(S)$ \cite{Owen}.  Additionally, it is assumed that the so-called empty coalition $\emptyset$ produces no value, i.e. $v(\emptyset)=0$.

 One of the distinguished families of cooperative games are simple games \cite{Isbell}. A simple game $(N,\mathcal{W})$ is defined by a set $\mathcal{W}$, which is a "winning" set of subsets of players' set $N$, i.e. $\mathcal{W}\subset 2^N$. The set $\mathcal{W}$ fulfills three main properties:
\begin{itemize}
    \item $N\in \mathcal{W},$
   \item $\emptyset\not\in \mathcal{W},$
   \item $(S\subseteq T\subseteq N\ \text{ and }\ S\in \mathcal{W})\Rightarrow T\in \mathcal{W}.$
\end{itemize}
Clearly, this game can be represented as a standard cooperative game via simple identification $(N,\mathcal{W})\sim(N,v),$ where  $$v(S)=\begin{cases}1,\ \text{if} \ S\in \mathcal{W}\\
0,\ \text{ otherwise.}
\end{cases}$$
For the purpose of this paper, we work with the class of the weighted majority games \cite{Neyman}, where there exist quota $q$ and weights $w_i$ associated with each player $i\in N$. Then, we can identify weighted majority game $\left(N,q,(w_i)_{i\in N}\right)$ with the simple game $(N,\mathcal{W})$ using the following relation 
$$S\in \mathcal{W}\Leftrightarrow \sum_{i\in S} w_i\geq q.$$
To "solve" a game often means to compute its solution. In terms of cooperative game theory, a solution is defined as a function $\sigma$ that assigns each game $(N,v)$ a subset (or just one) of allocations $\sigma(N,v)$ from the set $X^\ast(N,v)$ 
\begin{equation*}
X^\ast(N,v)=\{(x_i)_{i\in N} |\ \sum_{i\in N}x_i\leq v(N)\},   
\end{equation*}
of all feasible allocations of the $v(N)$, where player's $i$ payoff is denoted as $x_i$ \cite{Peleg}. One of the main one-point solution concepts known in game theory is the Shapley value. The Shapley value for player $i\in N$ can be defined as 
$$\phi_i (N,v)=\sum_{S\subseteq N\setminus \{i\}}\frac{|S|!(|N|-|S|-1)!}{|N|!} (v(S\cup\{i\})-v(S)). $$
The shapley value assigns to each player his or her expected payoff in the following situation. Assume players arrive randomly and each order has an equal probability of $\frac{1}{|N|!}$. Then, when a player arrives, he or she obtains the marginal contribution to the coalition of the already arrived players \cite{Peleg}.

Before the Shapley value concept for simple games is introduced, the property of a player to be pivotal to the coalition has to be explained. In a simple game $(N,\mathcal{W})$, a player $i\in N$ is pivotal to some coalition $S$, iff $S\not\in\mathcal{W}$, but $S\cup\{i\}\in \mathcal{W}$. The analogy of the Shapley value for simple games is called the Shapley-Shubik power index \cite{ShapleyShubik}. It assigns to the player's percentual share of how many times the player is pivotal to a coalition of already arrived players in each sequential coalition (possible ordering of players). This power index reflects the player's power or pretension. Now, two particular instances of the weighted majority games will be presented and solved. 
 \begin{exm} \label{priklad1} Consider weighted majority game $\left(N,q,(w_i)_{i\in N}\right)$, where $N=\{1,2\}$, $w_1=1,w_2=1$, and $q=1$. Table 1 demonstrates players' property of being pivotal (indicated by 1) for each possible order of arrival.
\begin{table}[h!]
\begin{center}
\begin{tabular}{ccc}
Sequential coalition         & Player 1 & Player 2  \\ \hline
\multicolumn{1}{c|}{(1,2)} & 1        & 0            \\
\multicolumn{1}{c|}{(2,1)} & 0        & 1        
\end{tabular}
\caption{Marginal contributions of two players to sequential coalitions.}
\end{center}
\end{table}

Thus, if we denote the Shapley-Shubik power index of player $i$ as $\phi_i$, we have the following results: $\phi_1\left(N,q,(w_i)_{i\in N}\right)=50\%,\ \phi_2\left(N,q,(w_i)_{i\in N}\right)=50\%$. In this game, players obtain the same distribution from the creation of the joint utility, though each of them exceeds the quota on one's own. 
 \end{exm}
 
 \begin{exm} \label{priklad2} Consider weighted majority game $\left(N,q,(w_i)_{i\in N}\right)$, where $N=\{1,2,3\}$, $w_1=1,w_2=2,w_3=1$, and $q=2$. 
 Table 2 demonstrates players' property of being pivotal for each possible order of arrival.
\begin{table}[h!]
\begin{center}
\begin{tabular}{cccc}
Sequential coalition         & Player 1 & Player 2 & Player 3 \\ \hline
\multicolumn{1}{c|}{(1,2,3)} & 0        & 1        & 0        \\
\multicolumn{1}{c|}{(1,3,2)} & 0        & 0        & 1        \\
\multicolumn{1}{c|}{(2,1,3)} & 0        & 1        & 0        \\
\multicolumn{1}{c|}{(2,3,1)} & 0        & 1        & 0        \\
\multicolumn{1}{c|}{(3,1,2)} & 1        & 0        & 0        \\
\multicolumn{1}{c|}{(3,2,1)} & 0        & 1        & 0       
\end{tabular}
\caption{Marginal contributions of three players to sequential coalitions.}
\end{center}
\end{table}

Thus, we have the following results: $$\phi_1\left(N,q,(w_i)_{i\in N}\right)=16.6\overline{6}\%,\phi_2\left(N,q,(w_i)_{i\in N}\right)=66.6\overline{6}\%,\phi_3\left(N,q,(w_i)_{i\in N}\right)=16.6\overline{6}\%.$$ Clearly, the second player dominates since he is pivotal 4 times out of 6 possible. 
 
 \end{exm}
 
However, the results of both games are true only for the ideal deterministic case, when a grand coalition $N$ will always be formed and all players want to cooperate and are able to come to the agreement. On the other side, what if we allow players to form some bonds and pre-agreements or have doubts about cooperation?  Moreover, what if this information will affect the probabilities of coalitions' occurrence and we will distribute players' expected utilities according to the restricted Shapley values? Will it lead to the redistribution of power indices and change the situation? To answer all these questions, we propose to quantize the considered weighted majority games. This will enable us to fairly distribute payoffs on the basis of additional information about players' relationships and intentions. 

\section{Cooperative games quantization}
The main principle of the proposed quantization of cooperative games is the idea that a number of basis states representing $n$--qubit exactly corresponds to a cardinality of power set $2^N$ of players' set $N=\{1,...,n\}$. Thus, we can easily identify each basis state with the particular coalition $S\subseteq N$ in the following way: $$S\sim\ket{a_1...a_i...a_n} \Leftrightarrow a_i=\begin{cases}
1,\ i\in S\\
0,\ i\not\in S
\end{cases}$$

Therefore, it is possible to associate the probability of occurrence of each basis state with the probability of occurrence of a particular coalition. Then, we propose to distribute players' payoffs (power indices) as an expected value of their Shapley values computed for the  value functions restricted on the coalitions $S\subseteq N$ with respect to new quantum probabilities. 

Our main objective is to study how information about players' initial agreement and their satisfaction with this pre-agreement can redistribute the resulting payoffs through this quantum value. In the following subsection, we will introduce new parameters reflecting the additional information about the game setting and demonstrate how they affect the probabilities of the basis states in two-player games.

\subsection{Quantization of two-player game}
At first, we focus on two-player cooperative games to better demonstrate the principles of the quantum cooperative games. Moreover, this simple instance will enable us to show that the proposed approach preserves and extends classical cooperative games. The proposed game quantization scheme is depicted in Figure \ref{fig:rep_pris3}.
\begin{figure}[h!]
  \centering
  \includegraphics[scale=0.9]{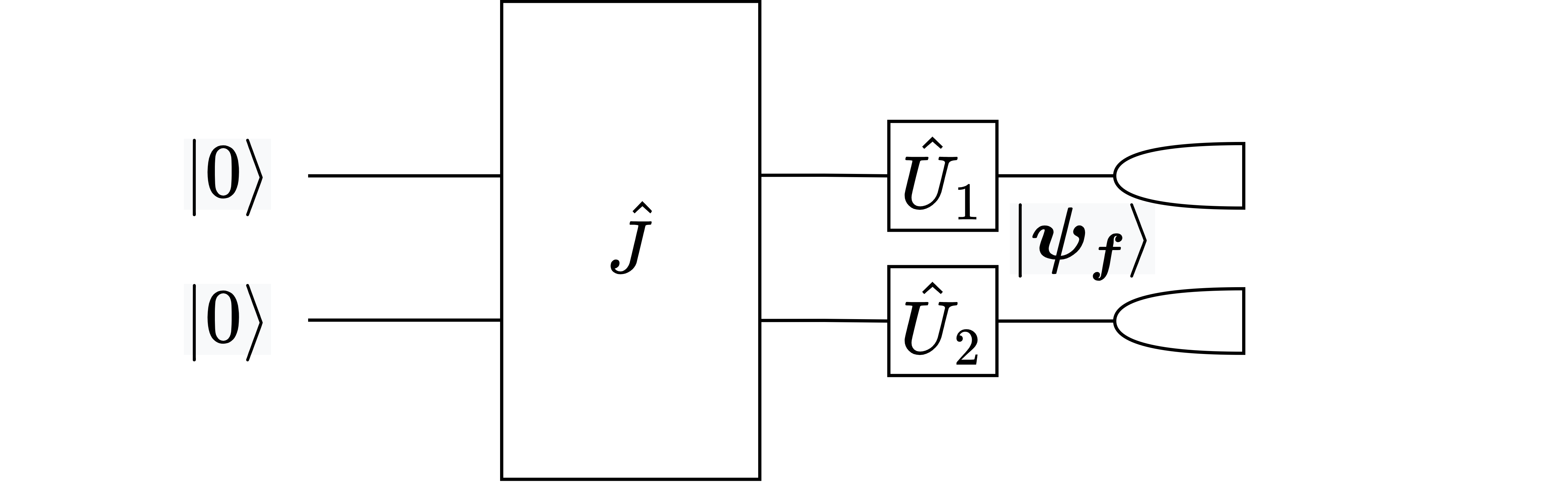}
  \caption{Two-player protocol scheme.}
  \label{fig:rep_pris3}
\end{figure}

The scheme almost fully corresponds to the quantization of EWL protocol \cite{Eisert}, but does not end with the disentanglement operator before the measuring device. The consideration that bonds, created by the entanglement, have to be preserved in order to take them into consideration during the payoff distribution explains the absence of the disentanglement operator. Moreover, for the certain choices of operators $\hat{U}_1(p_1)$ and $\hat{U}_2(p_2)$, entanglement may not play a role in the $\ket{\psi_f}$, if the disentanglement operator $\hat{J}^\dagger$ is applied. 

The presence of each player from $N=\{1,2\}$ in the game is initially represented by the basis state $\ket{0}$. It describes the fact that the player does not cooperate. Thus, the initial 2--qubit state $\ket{00}$ corresponds to the empty coalition $\emptyset$, whereas for the remaining basis states the identification 
$$\ket{01}\sim\{2\},\ \ket{10}\sim\{1\}, \ket{11}\sim\{1,2\},$$
holds. Then, the entanglement \cite{W} operator $\hat{J}$ is applied on the $2$--qubit state $\ket{00}$. In particular, we assume an entanglement operator \begin{equation*}
    \hat{J}(\gamma)= \begin{pmatrix}
\cos{\frac{\gamma}{2}} & 0 & 0 & i\sin{\frac{\gamma}{2}} \\
0 & \cos{\frac{\gamma}{2}} & -i\sin{\frac{\gamma}{2}} & 0 \\
0 & -i\sin{\frac{\gamma}{2}} & \cos{\frac{\gamma}{2}} & 0 \\
i\sin{\frac{\gamma}{2}} & 0 & 0 & \cos{\frac{\gamma}{2}} \\
\end{pmatrix}, 
\end{equation*}
presented in the EWL protocol \cite{Eisert}. The operator  $\hat{J}(\gamma)$ depends on the entanglement measure $\gamma\in [0,\pi/2]$, where $\hat{J}(0)$ is an identity operator and $\hat{J}(\pi/2)$ creates one of maximally entangled two-qubit states $(\ket{00}+\ket{11})/\sqrt{2},$ called the Bell's state. In this work, the entanglement parameter $\gamma$ is interpreted as a measure of the initial agreement between players. When players are maximally entangled, the Bell's state is created and
 probability of occurrence of coalition $N$ is $1/2$ and the same is valid for  $\emptyset$. Thus, the full pre-agreement between players affects the game's possible outcome such that both of them are either together or do not participate in the game at all. 

 After that, the information about players' desire to change the initial state is incorporated into the game using tensor products of unitary operators $\hat{U}_i$. We interpret the operator $\hat{U}_i(p_i)$ as the player's $i\in N$ satisfaction with the outcome after the initial agreement and can be defined as follows $$\hat{U}_i(p_i)=\begin{pmatrix}
\cos(p_i) & \sin(p_i) \\
-\sin(p_i) & \cos(p_i) 
\end{pmatrix},$$ 
where $p_i\in[0,\pi/2], \forall i\in N$. This strategy operator corresponds to the one considered in \cite{Elgazzar}. The greater $p_i$ indicates greater will to intervene in the game process and change the initial state. 
The final state of the game is then 
\begin{equation}
\label{psif}
\ket{\psi_f}=\left(\hat{U}_1(p_1)\otimes \hat{U}_2(p_2)\right)\hat{J}(\gamma)\ket{00}.    
\end{equation}
Thus, a two-player quantum cooperative game associated with the classical cooperative game $(N,v)$, $|N|=2,$ can be defined as $(N,v,\gamma,p_1,p_2)$, where $\gamma\in [0,\pi/2]$ and $p_i\in[0,\pi/2], \forall i\in N$. Now, we can proceed to the definition of the quantum Shapley value. 

\subsubsection*{\textbf{The quantum Shapley value of a two-player quantum cooperative game}}
We define the quantum Shapley value $(\tilde{\phi_i})_{i\in N}$ of the two-player quantum cooperative game $(N,v,\gamma,p_1,p_2)$ as
$$\tilde{\phi}_i(N,v,\gamma,p_1,p_2)=\sum_{S\subseteq N: i\in S}p(S)\phi_i(S,v_S),$$
where $v_S$ is a restriction of initial value function over coalition $S$ and $p(S)$ is a probability of occurrence of coalition $S$ defined as $$p(S)=|\braket{a_1a_2|\psi_f}|^2,$$
where $S\sim\ket{a_1a_2}$ and $\psi_f$ is calculated according to \eqref{psif}.
Table \ref{hodnoty1}, Table \ref{hodnoty2}, and Table \ref{hodnoty3}  demonstrate how different boundary values of the newly introduced parameters will affect the final distribution provided by the quantum Shapley value for player $i=1$. All numbers were rounded to three decimal places. 

\begin{table}[h!]
\begin{center}
\begin{tabular}{c|ccc}
\small
${p_1}$\textbackslash ${\ p_2}$    & 0   & $\pi/4$   & \multicolumn{1}{c}{$\pi/2$ }  \\ \hline
0        & 0   & 0     & 0       \\
$\pi/4$      & 0.5$v(\{1\})$ & 0.25$v(\{1\})$+0.25$\phi_1(N,v)$ & 0.5$\phi_1(N,v)$                   \\
$\pi/2$        & $v(\{1\})$   & 0.5$v(\{1\})$+0.5$\phi_1(N,v)$ & $\phi_1(N,v)$                    
\end{tabular}
\caption{Quantum Shapley value $\tilde{\phi_1}$ for $\gamma=0$.}
\label{hodnoty1}
\end{center}
\end{table}
\begin{table}[h!]
\begin{center}
\begin{tabular}{c|ccc}
\small
${p_1}$\textbackslash ${\ p_2}$    & 0   & $\pi/4$   & \multicolumn{1}{c}{$\pi/2$ }  \\ \hline
0        & 0.146$\phi_1(N,v)$    & 0.073$v(\{1\})$+0.073$\phi_1(N,v)$     & 0.146$\phi_1(N,v)$       \\
$\pi/4$      & 0.427$v(\{1\})$+0.073$\phi_1(N,v)$ & 0.25$v(\{1\})$+0.25$\phi_1(N,v)$ & 0.073$v(\{1\})$+0.427$\phi_1(N,v)$                   \\
$\pi/2$        & 0.854$\phi_1(N,v)$   & 0.427$v(\{1\})$+0.427$\phi_1(N,v)$ & 0.854$\phi_1(N,v)$                    
\end{tabular}
\caption{Quantum Shapley value $\tilde{\phi_1}$ for $\gamma=\pi/4$.}
\label{hodnoty2}
\end{center}
\end{table}
\begin{table}[h!]
\begin{center}
\begin{tabular}{c|ccc}
\small
${p_1}$\textbackslash ${\ p_2}$    & 0   & $\pi/4$   & \multicolumn{1}{c}{$\pi/2$ }  \\ \hline
0        & 0.5$\phi_1(N,v)$   & 0.25$v(\{1\})$+0.25$\phi_1(N,v)$     & 0.5$v(\{1\})$       \\
$\pi/4$      & 0.25$v(\{1\})$+0.25$\phi_1(N,v)$ & 0.25$v(\{1\})$+0.25$\phi_1(N,v)$ & 0.25$v(\{1\})$+0.25$\phi_1(N,v)$                   \\
$\pi/2$        & 0.5$v(\{1\})$    & 0.25$v(\{1\})$+0.25$\phi_1(N,v)$ & 0.5$\phi_1(N,v)$                    
\end{tabular}
\caption{Quantum Shapley value $\tilde{\phi_1}$ for $\gamma=\pi/2$.}
\label{hodnoty3}
\end{center}
\end{table}
As we can see, the proposed approach recreates the original solution $\phi(N,v)$ for the choice $\gamma=0,p_1=p_2=\pi/2$. Thus, it extends the classical Shapley value but also brings other, non-trivial, outcomes. Depending on the definition of the underlying value function $v$, the player might achieve a greater or lesser payoff in the quantum setting compared to the canonical one. For example, in case $v(N)>v(1)+v(2)$, every deviation from the original classical scenario will penalize the player. Thus, the pre-agreement bonds only worsen the position of the player, making the player more vulnerable to a possible betrayal. Alternatively, if $v(N)\leq v(1)+v(2)$ holds, the player is able to gain an advantage and raise greater claims.  In Example \ref{priklad1}, in case $\gamma=\pi/4$, $p_1=\pi/2$, $p_2=\pi/4$ player $i=1$ achieves $\tilde{\phi}_1(N,v,\gamma,p_1,p_2)=64\%$ instead of original ${\phi}_1(N,v)=50\%$.  Now, three-player games will be quantized in the next subsection. 

\subsection{Quantization of three-player game}
Three-player cooperative games can be quantized analogically to the two-player games. The proposed scheme is depicted in Figure 2. \begin{figure}[h!]
  \centering
  \includegraphics[scale=0.7]{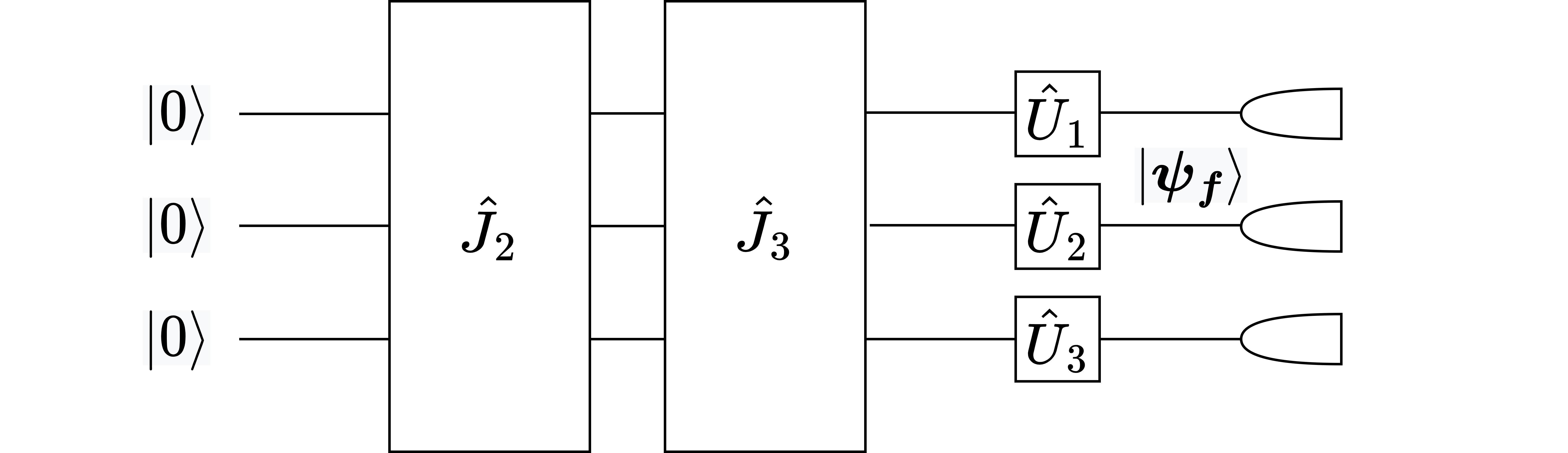}
  \caption{Three-player protocol scheme.}
\end{figure} 

The main novelty, compared to the previous scheme from Figure 1, is that there are two gates describing the entanglement: $\hat{J_2}$ and $\hat{J_3}$. The first entanglement gate creates entanglement between pairs of  qubits and can be represented as  
$$\hat{J_2}(\gamma_{12},\gamma_{13},\gamma_{23})=\left(\text{SWAP}\otimes \text{Id}\right)\left(\text{Id}\otimes \hat{J}(\gamma_{13}) \right)\left(\text{SWAP}\otimes \text{Id}\right) \left( \text{Id}\otimes \hat{J}(\gamma_{23})\right) \left(\hat{J}(\gamma_{12})\otimes  \text{Id}\right),$$
where parameter $\gamma_{ij}\in [0,\pi/2]$ describes entanglement between $i$-th and $j$-th player,
$$\text{Id}=\begin{pmatrix} 1 & 0\\ 0& 1\end{pmatrix}$$ is an identity operator, and 
$$\text{SWAP}=\begin{pmatrix} 1 & 0 & 0 & 0\\ 0 & 0 & 1 & 0\\ 0 & 1 & 0 & 0 \\0 & 0 & 0 & 1\end{pmatrix}$$ interchanges the input states. This entanglement gate can be interpreted as the one that takes into consideration the pre-agreement between pairs of players but does not consider a potential bond that can be created between all three of them at once. Exactly due to this reason we assume a second entanglement gate.

 Gate $\hat{J_3}(\gamma_{123})$ \cite{W} is assumed to create the so-called GHZ state \cite{Green}, when full entanglement between players is considered. 
 To avoid the parametrization of this gate, we consider only discrete possibilities for entanglement measure: $\gamma_{123}\in\{0,1\}$. 
 In case of zero entanglement, $\hat{J_3}(\gamma_{123})$  should act as an identity.  Then, we define $\hat{J_3}(\gamma_{123})$ as follows:
 
$$\hat{J_3}(\gamma_{123})=\begin{cases}
         \text{Id}\otimes \text{Id}\otimes \text{Id}, \text{ for }\gamma_{123}=0,\\
        (\text{Id}\otimes \text{\text{CNOT}}) (\text{\text{CNOT}}\otimes \text{Id}) (\text{H}\otimes \text{Id}\otimes \text{Id}), \text{ for }\gamma_{123}=1,
\end{cases}$$  
 where $$\text{H}=\frac{1}{\sqrt{2}}\begin{pmatrix}
1& 1 \\
1 & -1 
\end{pmatrix}$$
is a Hadamard gate and 
$$\text{\text{CNOT}}=\begin{pmatrix}
1& 0 &  0 & 0 \\
0& 1 & 0 & 0\\
0 & 0 & 0 & 1\\
0 & 0 & 1 & 0
\end{pmatrix}$$ is a controlled NOT gate. 
Thus, $\hat{J_3}(1)$ produces the GHZ state $$(\text{Id}\otimes \text{\text{CNOT}}) (\text{\text{CNOT}}\otimes \text{Id}) (\text{H}\otimes \text{Id}\otimes \text{Id})\ket{000}=\frac{\ket{000}+\ket{111}}{\sqrt{2}}.$$

Then, the final state of the three-player game is given as follows: 
\begin{equation}
\ket{\psi_f}=\left(\hat{U}_1(p_1)\otimes \hat{U}_2(p_2) \otimes \hat{U}_3(p_3)\right)\hat{J_3}(\gamma_{123})\hat{J_2}(\gamma_{12},\gamma_{13},\gamma_{23})\ket{000}. \label{3hraci}\end{equation}

Thus, analogically to a two-player game, a three-player quantum cooperative game associated with the classical cooperative game $(N,v)$, $|N|=3,$ can be defined as $(N,v,\gamma_{123}, \gamma_{12}, \gamma_{13},\gamma_{23} ,p_1,p_2, p_3)$, where $\gamma_{123}\in\{0,1\},$ $\gamma_{i,j}\in [0,\pi/2]$ and $p_i\in[0,\pi/2], \forall i, j\in N$. Then, the corresponding quantum Shapley value is
$$\tilde{\phi}_i(N,v,\gamma_{123}, \gamma_{12}, \gamma_{13},\gamma_{23} ,p_1,p_2, p_3)=\sum_{S\subseteq N: i\in S}p(S)\phi_i(S,v_S),$$  with $p(S)=|\braket{a_1a_2a_3|\psi_f}|^2,$ $S\sim\ket{a_1a_2a_3}$ and $\psi_f$ from \eqref{3hraci}. 

The Shapley value, is symmetric by its axiomatic definition. Therfore, in order to demonstrate that the proposed approach is reasonable, it is necessary to prove that the resulting entangled state $\hat{J_2}(\gamma_{12},\gamma_{13},\gamma_{23})\ket{000}$ does not depend on the ordering of the entanglement gates within $\hat{J_2}(\gamma_{12},\gamma_{13},\gamma_{23})$. However, in the three-player game direct calculations become rather extensive and complex. Therefore, in the next section, we will demonstrate how the formal language of QRA will help us to perform quantum computing and study the properties of the considered entanglement gate.

\section{Automated proofs of circuit equivalence based on quantum register algebra}

In this section, it is demonstrated that the ordering of the entanglements within $\hat{J_2} (\gamma_{12},\gamma_{13},\gamma_{23})$ is insignificant. We also indicate how the QRA apparatus can be used to perform the automated proofs of this and other similar properties. 

\subsection{Quantum Register Algebra}
Consider a  geometric algebra 
$\mathbb G_{2n}$ with the set of basis elements  $\{ e_1, \dots , e_{2n} \}$. This algebra can be seen as a subalgebra of geometric algebra $\mathbb G_{2n+2}$ with the set of basis elements $\{ e_1, \dots , e_{2n}, r_1,r_2 \}$. Then, we define QRA(n) \cite{Hrdina2023} as a geometric algebra $\mathbb G_{2n}$ with the coefficients from $\tilde{\mathbb C} = \{ a+b \iota | a,b \in \mathbb R, \iota = r_1 r_2 \}$, i.e.
$$\mathrm{QRA(n)} = \{ a_1 g_1 + \cdots + a_{2n} g_{2n}  | a_i \in \tilde{\mathbb  C}, g_i \in \mathbb G_{2n} \} .$$

An important part of this construction is the definition of QRA conjugation as a Hermitean-linear anti-automorphism that extends identity on vectors \cite{hrdina2022quantum}, i.e.    
$$(a_1 g_1 + \cdots + a_{2n} g_{2n})^{\dagger} = 
\bar a_1 g_1^{\dagger} + \cdots + \bar a_{2n} g_{2n}^{\dagger}, $$
where $\bar a_i= \overline{a+b \iota } = a-b \iota \in \tilde{\mathbb C}$, 
$ (ab)^{\dagger}=b^{\dagger}a^{\dagger}$ and $e_i^{\dagger}=e_i$. To use QRA to model quantum computing, we choose a different basis. This basis is called Witt basis and is formed by elements  
\begin{align*}
f_i= \frac{1}{2} (e_i + \iota e_{i+n}), \: \: \:
f_i^\dagger =\frac{1}{2} (e_i - \iota e_{i+n}), \: i=1,\dots,n,  \end{align*} 
where the rules for computation with the Witt basis are given as follows:
\begin{align*}
(f_i)^2&=(f_i^{\dagger})^2=0 ,\quad
f_i f_j = -f_j f_i ,  \quad
f_i^{\dagger} f_j^{\dagger} = -f_j^{\dagger} f_i^{\dagger},   \\
f_i f_i^{\dagger} f_i&=f_i,   \qquad
f_i^{\dagger} f_i f_i^{\dagger}=f_i^{\dagger}, \qquad f_i^{\dagger} f_j = - f_j f_i^{\dagger}.
\end{align*}

There is the following straightforward identification  of bra and ket vectors of Dirac formalism with elements of QRA:
\begin{align*} 
\langle a_1 \dots a_n| \longleftrightarrow 
I(f_n)^{a_n}
\dots (f_1)^{a_1}, \text{ where }a_i \in \{ 0,1\},\\
| a_1 \dots a_n \rangle \longleftrightarrow 
(f_1^{\dagger})^{a_1} 
\dots (f_n^{\dagger})^{a_n} I, \text{ where }a_i \in \{ 0,1\},
\end{align*}
where  $I=f_1f_1^{\dagger} \cdots f_nf_n^{\dagger}$. To describe a three-player game, the space of $3$--qubit states should be considered, i.e. we will work with the identification 
\begin{align*}
    | 000 \rangle & \longleftrightarrow  (f_1^{\dagger})^{0} (f_2^{\dagger})^{0} (f_3^{\dagger})^{0} I= I, \\
    | 001 \rangle  &\longleftrightarrow (f_1^{\dagger})^{0} (f_2^{\dagger})^{0} (f_3^{\dagger})^{1} I = f_3^{\dagger}I,
    \\
    | 010 \rangle & \longleftrightarrow   (f_1^{\dagger})^{0} (f_2^{\dagger})^{1} (f_3^{\dagger})^{0} I=  f_2^{\dagger}I, 
    \\
    | 011 \rangle  &\longleftrightarrow (f_1^{\dagger})^{0} (f_2^{\dagger})^{1} (f_3^{\dagger})^{1}I = f_2^{\dagger} f_3^{\dagger} I, 
    \\
    | 100 \rangle & \longleftrightarrow   (f_1^{\dagger})^{1} (f_2^{\dagger})^{0} (f_3^{\dagger})^{0} I=   f_1^{\dagger} I,    \\
    | 101 \rangle & \longleftrightarrow  (f_1^{\dagger})^{1} (f_2^{\dagger})^{0} (f_3^{\dagger})^{1}I =  f_1^{\dagger} f_3^{\dagger} I,
    \\
    | 110 \rangle & \longleftrightarrow  (f_1^{\dagger})^{1} (f_2^{\dagger})^{1} (f_3^{\dagger})^{0} I= f_1^{\dagger} f_2^{\dagger} I, \\  
    | 111 \rangle  & \longleftrightarrow  (f_1^{\dagger})^{1} (f_2^{\dagger})^{1}(f_3^{\dagger})^{1} I = f_1^{\dagger} f_2^{\dagger} f_3^{\dagger} I.
\end{align*}
The space of $3$--qubit bra vectors can be also described by QRA conjugation as $ \langle a_1a_2a_3 | =  | a_1a_2a_3 \rangle^{\dagger} $, for example
$$ \langle 111 |= | 111 \rangle^{\dagger}  \longleftrightarrow   (f_1^{\dagger} f_2^{\dagger} f_3^{\dagger})^{\dagger} I = f_3 f_2 f_1 I. $$ The analogical identification on two-qubit states has been already presented in \cite{Eryganov}. 

We conclude this subsection by answering the question about how a circuit can be composed of individual blocks.  According to \cite{Hrdina2023}, a serial circuit, formed by sequential application of the gates $f_A$ and $f_B$ on the same qubit, can be represented in QRA as $f_Bf_A$. A parallel circuit \cite{Hrdina2023}, consisting of the gates $f_A$ and $f_B$, acting on different qubits, can be represented as $f_Af_B$ up to a sign of individual monomials. When working with the two gates in a parallel circuit (in fact, their tensor product), it is sufficient to perform the following procedure, that has been originally presented in \cite{Hrdina2023}. 
\begin{enumerate}
\item On the right side of multiplication, we assign artificial coefficient $b$ to the monomials with the odd number of
terms. 
\item On the left side of multiplication,  $a$ is assigned to monomials with the odd number of occurrences of elements of type $f_i^\dagger f_i$ or $f_i$. 
\item Then, after the multiplication, we perform simple reassignment:
$ab\rightarrow-1,\ a,b\rightarrow1.$
\end{enumerate}
Since multiplication in QRA is an associative operation, it is sufficient to apply the above-defined rule subsequently on pairs of quantum gates. This technical step is necessary due to the nature of the problem. On the other hand, it allows us for rather straightforward implementation. Examples of the serial and parallel circuits constructed via QRA can be found in \cite{Hrdina2023}.

\subsection{Circuit identities}
At first, to highlight the convenience of QRA notation, we use this algebra to find general representatives of $\text{SWAP}(s,t)$ gates, describing the interchange of qubits $s$ and $t$. At the end of this subsection, we apply QRA and GAALOP to perform the automated proof of the irrelevance of the entanglement gates ordering within $\hat{J_2}(\gamma_{12},\gamma_{13},\gamma_{23})$. 
\begin{lem} \label{l:sw12}
Let $\ket{\psi}$ be an $n$--qubit and $1\leq s < n$, $s \in \mathbb Z$. Then, the element 
\begin{align} 
\label{swap12}
\text{SWAP}(s,s+1)=f_s  f_s^{\dagger} f_{s+1}  f_{s+1}^{\dagger} - f_{s}f_{s+1}^{\dagger}+ f_{s}^{\dagger}f_{s+1} + f_{s}^{\dagger} f_{s} f_{s+1}^{\dagger}f_{s+1}  
\end{align}
 acts on $\ket{\psi}$ as a $\text{SWAP}$ between qubits $s$ and ${s+1}$.  
\end{lem}
\begin{proof}
Let us note that the elements 
$f_s  f_s^{\dagger} f_{s+1}  f_{s+1}^{\dagger}$
and $f_{s}^{\dagger} f_{s} f_{s+1}^{\dagger}f_{s+1}$ act as identities, if they act nontrivially (result is not zero). The 
elements $f_{s}f_{s+1}^{\dagger}$ and $f_{s}^{\dagger}f_{s+1} $ interchange two adjacent elements. We will use this property frequently throughout this section.
The following direct computations are based on the fact that all elements of \eqref{swap12} have an even number of monomials. \begin{align*}
& 
\text{SWAP}(s,s+1) (f_1^{\dagger})^{a_1} 
\cdots (f_{n}^{\dagger})^{a_n}I
= 
(f_1^{\dagger})^{a_1} 
\cdots (f_{s-1}^{\dagger})^{a_{s-1}}
\text{SWAP}(s,s+1)
(f_s^{\dagger})^{a_s} 
\cdots (f_{n}^{\dagger})^{a_n} I\\
& = 
(f_1^{\dagger})^{a_1} 
\cdots (f_{s-1}^{\dagger})^{a_{s-1}}
\Big [ \text{SWAP}(s,s+1)
(f_s^{\dagger})^{a_s}(f_{s+1}^{\dagger})^{a_{s+1}} 
\Big ]
(f_{s+2}^{\dagger})^{a_{s+2}} 
\cdots (f_{n}^{\dagger})^{a_{n}}I
\end{align*}
and  $\text{SWAP}(s,s+1)$ acts on $(f_s^{\dagger})^{a_s}(f_{s+1}^{\dagger})^{a_{s+1}}$  as $\text{SWAP}$ which completes the proof.
\end{proof}

\begin{thm} 
Let $\ket{\psi}$ be an $n$--qubit. Then, the element 
\begin{align}
\begin{split} \label{swap}
\text{SWAP}(s,t)&= 
(f_s  f_s^{\dagger} f_t  f_t^{\dagger}- f_s^{\dagger}f_t- f_sf_t^{\dagger}- f_s^{\dagger}f_sf_t^{\dagger}f_t ) \\
&(\sum_{ \sum (a_i) \text{ is odd } }  (f_{s+1}^{\dagger}f_{s+1})^{a_{s+1}} (f_{s+1} f_{s+1}^{\dagger})^{b_{s+1}} \cdots (f_{t-1}^{\dagger}f_{t-1})^{a_{t-1}} (f_{t-1} f_{t-1}^{\dagger})^{b_{t-1}})
\\
+&
(f_s  f_s^{\dagger} f_t  f_t^{\dagger} + f_s^{\dagger}f_t- f_s f_t^{\dagger}+ f_s^{\dagger}f_sf_t^{\dagger}f_t  ) \\
&(\sum_{ \sum (a_i) \text{ is even} }  (f_{s+1}^{\dagger}f_{s+1})^{a_{s+1}} (f_{s+1} f_{s+1}^{\dagger})^{b_{s+1}} \cdots 
(f_{t-1}^{\dagger}f_{t-1})^{a_{t-1}} (f_{t-1} f_{t-1}^{\dagger})^{b_{t-1}})
\end{split}
\end{align}
acts as a SWAP gate between $s^{\text{th}}$ and $t^{\text{th}}$ qubit ($s<t$).  \end{thm}
\begin{proof}
Because each part of the expression \eqref{swap} has an even number of elements, it is easy to show that 
$$\text{SWAP}(s,t) ( f_1^{\dagger})^{a_1} \cdots  (f_n^{\dagger})^{a_n} I =
( f_1^{\dagger})^{a_1} \cdots  (f_{s-1}^{\dagger})^{a_{s-1}}
\text{SWAP}(s,t) ( f_s^{\dagger})^{a_s} \cdots  (f_n^{\dagger})^{a_n}I 
$$
and, because of associativity, we have an expression
$$
\text{SWAP}(s,t) ( f_s^{\dagger})^{a_s} \cdots  (f_n^{\dagger})^{a_n} I =
[\text{SWAP}(s,t) ( f_s^{\dagger})^{a_s} \cdots  (f_t^{\dagger})^{a_t}]( f_{t+1}^{\dagger})^{a_{t+1}} \cdots  (f_n^{\dagger})^{a_n} I.
$$
Thus, without loss of generality, we can only discuss the gate 
\begin{align}
\begin{split} \label{s1n}
\text{SWAP}(1,n)&=
(f_1  f_1^{\dagger} f_n  f_n^{\dagger} - f_1^{\dagger}f_n-f_1f_n^{\dagger} - f_1^{\dagger}f_1f_n^{\dagger}f_n  ) \\
&(\sum_{ \sum (a_i) \text{ is odd } }  (f_2^{\dagger}f_2)^{a_2} (f_2 f_2^{\dagger})^{b_2} \cdots 
(f_{n-1}^{\dagger}f_{n-1})^{a_{n-1}} (f_{n-1} f_{n-1}^{\dagger})^{b_{n-1}})
\\
+&
(f_1  f_1^{\dagger} f_n  f_n^{\dagger} + f_1^{\dagger}f_n-f_1f_n^{\dagger} + f_1^{\dagger}f_1f_n^{\dagger}f_n  ) \\
&(\sum_{\sum (a_i) \text{ is even} }  (f_2^{\dagger}f_2)^{a_2} (f_2 f_2^{\dagger})^{b_2} \cdots 
(f_{n-1}^{\dagger}f_{n-1})^{a_{n-1}} (f_{n-1} f_{n-1}^{\dagger})^{b_{n-1}})
\end{split}
\end{align} 
Let $\ket{\psi}$ be an $n$--qubit, if $a_i=b_i =1$  then  $f_i^{\dagger} f_i f_i f_{i}^{\dagger} =0$. Therefore, in the expressions \eqref{swap} and  \eqref{s1n}, there are only such elements that $a_i+b_i \in \{0,1\}$. The gate  
$$(f_2^{\dagger}f_2)^{a_2} (f_2 f_2^{\dagger})^{b_2} \cdots 
(f_{n-1}^{\dagger}f_{n-1})^{a_{n-1}} (f_{n-1} f_{n-1}^{\dagger})^{b_{n-1}}$$
acts as the projection to the state 
\begin{align*}
&\sum_{a_1,a_n \in \{0,1\}} \psi_{a_1\dots a_n} (f_1^{\dagger})^{a_1} ( f_2^{\dagger})^{a_2} \cdots  (f_{n-1}^{\dagger})^{a_{n-1}}(f_n^{\dagger})^{a_n}I \\
&=
\psi_{0a_2\dots a_{n-1}0}  ( f_2^{\dagger})^{a_2} \cdots  (f_{n-1}^{\dagger})^{a_{n-1}}I+\psi_{0a_2\dots a_{n-1}1} ( f_2^{\dagger})^{a_2} \cdots  (f_{n-1}^{\dagger})^{a_{n-1}}f_n^{\dagger}I\\
&+\psi_{1a_2\dots a_{n-1}0} f_1^{\dagger} ( f_2^{\dagger})^{a_2} \cdots  (f_{n-1}^{\dagger})^{a_{n-1}}I+\psi_{1a_2\dots a_{n-1}1} f_1^{\dagger} ( f_2^{\dagger})^{a_2} \cdots  (f_{n-1}^{\dagger})^{a_{n-1}}f_n^{\dagger}I\\
&=\psi_{0a_2\dots a_{n-1}0}  ( f_2^{\dagger})^{a_2} \cdots  (f_{n-1}^{\dagger})^{a_{n-1}}I+
 (-1)^{\sum a_i}\psi_{0a_2\dots a_{n-1}1} f_n^{\dagger} ( f_2^{\dagger})^{a_2} \cdots  (f_{n-1}^{\dagger})^{a_{n-1}}I\\
&+\psi_{1a_2\dots a_{n-1}0} f_1^{\dagger} ( f_2^{\dagger})^{a_2} \cdots  (f_{n-1}^{\dagger})^{a_{n-1}}I+
(-1)^{\sum a_i}\psi_{1a_2\dots a_{n-1}1} f_1^{\dagger} f_n^\dagger( f_2^{\dagger})^{a_2} \cdots  (f_{n-1}^{\dagger})^{a_{n-1}}I
\end{align*}
Thus, if $\sum a_i$ is odd, the middle elements of the SWAP gate must have a different sign with respect to the corresponding projections. Then, because of Lemma \ref{l:sw12}, 
the element acts as a SWAP between the first and $n^{\text{th}}$ qubit which completes the proof.
\end{proof}

Finally, with the help of GAALOP \cite{alves20mathematica,hil3}, we demonstrate that the order of the entanglement operators within $\hat{J_2}(\gamma_{12},\gamma_{13},\gamma_{23})$ does not affect the outcome of the game. The following code

\begin{lstlisting}[caption={Code for the automated proof of the entanglement symmetry.}]
i = er1 * er2 ;
f1 = 0.5*( e1 + i * e4 );
f1T = 0.5*( e1 - i * e4 );
f2 = 0.5*( e2 + i * e5 );
f2T = 0.5*( e2 - i * e5 );
f3 = 0.5*( e3 + i * e6 );
f3T = 0.5*( e3 - i * e6 );
I = f1 * f1T * f2 * f2T* f3 * f3T ;
psi=I;
J12 = cos(gamma12/2)*(f1*f1T*f2*f2T+f1*f1T*f2T*f2+f1T*f1*f2*f2T
+f1T*f1*f2T*f2) +i*sin(gamma12/2)*(-f1*f2+f1*f2T-f1T*f2+f1T*f2T);
Id3 = f3*f3T + f3T*f3;
J23 = cos(gamma23/2)(f2*f2T*f3*f3T + f2*f2T*f3T*f3+f2T*f2*f3*f3T
+f2T*f2*f3T*f3) +i*sin(gamma23/2)*(-f2*f3+f2*f3T-f2T*f3+f2T*f3T);
J13 = cos(gamma13/2)(f2*f2T*f3*f3T + f2*f2T*f3T*f3+f2T*f2*f3*f3T
+f2T*f2*f3T*f3) +i*sin(gamma23/2)*(-f2*f3+f2*f3T-f2T*f3+f2T*f3T);
Id1 = f1*f1T + f1T*f1;
SWAP12 =( f1 * f1T * f2 * f2T )+( f1T * f2 ) 
-( f1 * f2T )+( f1T * f1 * f2T * f2 );
J1= J12 * Id3;
J2= Id1 * J23;
J33=Id1*J13;
J3= SWAP12 * J33 *SWAP12;
S1=J1*J2*J3;
S2=J2*J3*J1;
S3=J2*J3*J1;
?X1=S1-S2;
?X2=S1-S3;
?X3=S2-S3;
\end{lstlisting}
has the output 
\begin{lstlisting}[caption={Output for the automated proof of entanglement symmetry.}]
function [X1, X2, X3] = script()
end
\end{lstlisting}
which proves that, under every possible ordering of entanglement gates, $\hat{J_2}(\gamma_{12},\gamma_{13},\gamma_{23})$ is represented by the same element of QRA. Thus, the implementation of QRA in GAALOP can serve as an instrument to check the equivalence of the circuits. It can be seen as a viable alternative to symbolic calculations in other available languages. Moreover, the found QRA representation of SWAP$(s,t)$ can be used to generalize the proposed approach into the domain of $n$-player quantum cooperative games in the future. In the next section, we will demonstrate another possible application of QRA to quantum cooperative game theory.

\section{The QRA representation of two- and three-player quantum cooperative games}

In this section, we establish the parametric expressions describing the quantum Shapley values of the two- and three-player quantum cooperative games using QRA. The detailed deduction of QRA representations of $1$-- and $2$--qubit gates can be found in \cite{hrdina2022quantum}.

\subsection{Two-player game quantum Shapley value}
The considered scheme is analogous to the two-player quantum non-cooperative game (except for the disentanglement operator) presented in \cite{Eisert}. Therefore, using the considerations established in \cite{Eryganov}, we can obtain the final $2$--qubit state 
\begin{equation}
\label{psif2}
\begin{aligned}
\ket{\psi_f}=\left(\hat{U}_1(p_1)\otimes \hat{U}_2(p_2)\right) \hat{J}(\gamma)\ket{00}&=\cos{\frac{\gamma}{2}}(\cos(p_1)\cos(p_2)f_1f_1^\dagger f_2f_2^\dagger\\
&-\cos(p_1)\sin(p_2)f_1f_1^\dagger f_2^\dagger-\sin(p_1)\cos(p_2)f_1^\dagger f_2f_2^\dagger\\
&+\sin(p_1)\sin(p_2)f_1^\dagger f_2^\dagger)+i\sin{\frac{\gamma}{2}}(\sin(p_1)\sin(p_2)f_1f_1^\dagger f_2 f_2^\dagger\\
&+\sin(p_1)\cos(p_2)f_1f_1^\dagger f_2^\dagger+\cos(p_1)\sin(p_2) f_1^\dagger f_2f_2^\dagger\\
&+\cos(p_1)\cos(p_2) f_1^\dagger f_2^\dagger).
\end{aligned}
\end{equation}
Thus, the quantum Shapley value of the two-player game can be represented as 
\begin{equation*}
\begin{aligned}
\tilde{\phi}_1(N,v,\gamma,p_1,p_2)&=(\cos^2{\frac{\gamma}{2}}\sin^2(p_1) \cos^2(p_2)+\sin^2{\frac{\gamma}{2}} \cos^2(p_1)\sin^2(p_2))\phi_i(\{1\},v_{\{1\}})\\
&+(\cos^2{\frac{\gamma}{2}}\sin^2(p_1)\sin^2(p_2)+\sin^2{\frac{\gamma}{2}} \cos^2(p_1) \cos^2(p_2))\phi_i(N,v).\\
\end{aligned}
\end{equation*}

\begin{equation*}
\begin{aligned}
\tilde{\phi}_2(N,v,\gamma,p_1,p_2)&=(\cos^2{\frac{\gamma}{2}} \cos^2(p_1)\sin^2(p_2)+\sin^2{\frac{\gamma}{2}}\sin^2(p_1) \cos^2(p_2))\phi_i(\{2\},v_{\{2\}})\\
&+(\cos^2{\frac{\gamma}{2}}\sin^2(p_1)\sin^2(p_2)+\sin^2{\frac{\gamma}{2}} \cos^2(p_1) \cos^2(p_2))\phi_i(N,v).\\
\end{aligned}
\end{equation*}
Then, according to the definition of Shapley value, we have $$\phi_i(\{i\},v_{\{i\}})=v(\{i\}),\ \phi_i(N,v)=\frac{v(\{i\})}{2}+\frac{v(N)-v(N\setminus\{i\})}{2}.$$
Thus, we can obtain the following expression: 
\begin{equation}
\label{shapley2}
\begin{aligned}
\tilde{\phi}_i(N,v)&=(\cos^2{\frac{\gamma}{2}}\sin^2(p_i)\cos^2(p_{N\setminus\{i\}})+\sin^2{\frac{\gamma}{2}}\cos^2(p_i)\sin^2(p_{N\setminus\{i\}})v(\{i\})\\
&+(\cos^2{\frac{\gamma}{2}}\sin^2(p_1)\sin^2(p_2)+\sin^2{\frac{\gamma}{2}} \cos^2(p_1) \cos^2(p_2))(\frac{v(\{i\})}{2}+\frac{v(N)-v(N\setminus\{i\})}{2}).
\end{aligned}
\end{equation}
It is easy to verify, that the obtained expression is
in full accordance with the results presented in Tables \ref{hodnoty1}, \ref{hodnoty2}, and \ref{hodnoty3}. Now, the QRA representation of the three-player quantum cooperative game will be described in detail. 
 

 \subsection{Three-player game quantum Shapley value}
 Three-player game starts with the qubit
 $$\ket{000}=f_1f_1^\dagger f_2f_2^\dagger f_3f_3^\dagger.$$ Then, the first entanglement gate $\hat{J_2}(\gamma_{12},\gamma_{13},\gamma_{23})$ is applied. The gate $\hat{J_2}(\gamma_{12},\gamma_{13},\gamma_{23})$ represents a series of gates $(\text{SWAP}\otimes \text{Id})\left(\text{Id}\otimes\hat{J}(\gamma_{13})\right)(\text{SWAP}\otimes \text{Id}) \left( \text{Id}\otimes \hat{J}(\gamma_{23})\right) \left(\hat{J}(\gamma_{12})\otimes  \text{Id}\right)$, with each of them being tensor product of at least two gates. Further, we will use the notation  
 $$\left(\text{SWAP}(1,2)\otimes \text{Id}(3)\right)\left(\text{Id}(1)\otimes\hat{J}(\gamma_{13})\right)\left(\text{SWAP}(1,2)\otimes \text{Id}(3)\right) \left( \text{Id}(1)\otimes \hat{J}(\gamma_{23})\right) \left(\hat{J}(\gamma_{12})\otimes  \text{Id}(3)\right)$$
 to prevent possible ambiguity and specify on which qubits the gates are applied.
 The first gate in the series is 
$ \left(\hat{J}(\gamma_{12})\otimes  \text{Id}(3)\right)$, where
$$\text{\text{Id}}(3)=f_3f_3^\dagger+f_3^\dagger f_3.$$ This is a tensor product of one $2$--qubit gate and one $1$--qubit gate. However, when the identity operator is on the right side of the tensor product, no change in sign can occur and it is sufficient to directly rewrite such gate as a multiplication.
Thus, we directly obtain the expression
\begin{equation*}
\begin{aligned}
\hat{J}(\gamma_{12})\otimes  \text{Id}&= \cos{\frac{\gamma_{12}}{2}}(f_1 f_1^\dagger f_2 f_2 ^\dagger +f_1^\dagger f_1 f_2 f_2^\dagger+ f_1 f_1^\dagger f_2^\dagger f_2 +f_1^\dagger f_1f_2^\dagger f_2)(f_3f_3^\dagger+f_3^\dagger f_3)\\
&+i\sin{\frac{\gamma_{12}}{2}}(-f_1f_2+f_1f_2^\dagger-f_1^\dagger f_2+f_1^\dagger f_2^\dagger)(f_3f_3^\dagger+f_3^\dagger f_3).
\end{aligned}
\end{equation*}
The next gate is $ \text{Id}(1)\otimes \hat{J}(\gamma_{23})$. However, the gate $\hat{J}(\gamma_{23})$ cannot affect signs of monomials. Thus, the following representation can be obtained: 
 \begin{equation*}
\begin{aligned}
 \text{Id(1)}\otimes \hat{J}(\gamma_{23})&= \cos{\frac{\gamma_{23}}{2}}(f_1f_1^\dagger+f_1^\dagger f_1)(f_2 f_2^\dagger f_3 f_3 ^\dagger +f_2^\dagger f_2 f_3 f_3^\dagger+ f_2 f_2^\dagger f_3^\dagger f_3 +f_2^\dagger f_2f_3^\dagger f_3)\\
&+i\sin{\frac{\gamma_{23}}{2}}(f_1f_1^\dagger+f_1^\dagger f_1)(-f_2f_3+f_2f_3^\dagger-f_2^\dagger f_3+f_2^\dagger f_3^\dagger).
\end{aligned}
\end{equation*}
Then, the input states have to be interchanged using $\text{SWAP}(1,2)\otimes \text{Id}(3)$ to entangle the remaining pair of qubits. The $\text{SWAP}(1,2)$ gate can be represented as 
$$\text{SWAP}(1,2)=f_1 f_1^{\dagger}f_2 f_2^{\dagger}+ f_1^{\dagger} f_2 - f_1 f_2^{\dagger}
+ f_1^{\dagger} f_1 f_2^{\dagger} f_2,
$$
and, again, by straightforward multiplication we obtain
 \begin{equation*}
\begin{aligned}
\text{SWAP}(1,2)\otimes \text{Id}(3)&=f_1 f_1^{\dagger}f_2 f_2^{\dagger}f_3f_3^\dagger+ f_1^{\dagger} f_2f_3f_3^\dagger - f_1 f_2^{\dagger}f_3f_3^\dagger
+ f_1^{\dagger} f_1 f_2^{\dagger} f_2f_3f_3^\dagger\\
&+f_1 f_1^{\dagger}f_2 f_2^{\dagger}f_3^\dagger f_3+ f_1^{\dagger} f_2f_3^\dagger f_3 - f_1 f_2^{\dagger}f_3^\dagger f_3
+ f_1^{\dagger} f_1 f_2^{\dagger} f_2f_3^\dagger f_3.
\end{aligned}
\end{equation*}
After that, the gate $\text{Id}(1)\otimes\hat{J}(\gamma_{13}) $ is applied, which completely copies the gate $\text{Id}(1)\otimes\hat{J}(\gamma_{23}) $ with changed entanglement parameter. At last, to preserve the initial identification of qubits with coalitions, we interchange the states back using $\text{SWAP}(1,2)\otimes \text{Id}(3)$ once more time.
Thus, the whole effect of the $\hat{J_2}(\gamma_{12},\gamma_{13},\gamma_{23})$ on the initial state $\ket{000}$ can be described by the multiplication of the basis state $\ket{000}$ by the above-described gates in the corresponding order.  
Now, the entanglement gate $\hat{J_3}(\gamma_{123})$ has to be applied. Due to the discrete nature of the entanglement parameter $\gamma_{123}$, we have divided this section into smaller subsections describing each possible choice of the parameter $\gamma_{123}$ separately.

\subsubsection{Case $\gamma_{123}=0$ and action of $1$--qubit gates}
In case $\gamma_{123}=0$, the operator $\hat{J_3}(\gamma_{123})$ collapses into 
$\hat{J_3}(0)= \text{Id}(1)\otimes \text{Id}(2)\otimes \text{Id}(3)$. Thus, it can be described by the following element of QRA 
 \begin{equation*}
\begin{aligned}
\hat{J_3}(0)&=f_1f_1^\dagger f_2f_2^\dagger f_3f_3^\dagger+ f_1f_1^\dagger f_2f_2^\dagger f_3^\dagger f_3+f_1f_1^\dagger f_2^\dagger f_2 f_3f_3^\dagger+ f_1f_1^\dagger f_2^\dagger f_2  f_3^\dagger f_3\\
&+f_1^\dagger f_1 f_2f_2^\dagger f_3f_3^\dagger+ f_1^\dagger f_1 f_2f_2^\dagger f_3^\dagger f_3+f_1^\dagger f_1 f_2^\dagger f_2 f_3f_3^\dagger+ f_1^\dagger f_1 f_2^\dagger f_2  f_3^\dagger f_3,
\end{aligned}
\end{equation*}
which does not affect the state $\hat{J_2}(\gamma_{12},\gamma_{13},\gamma_{23})\ket{000}$.

Therefore, before the application of unitary operators,
 we obtain the state 
\begin{equation*}
\begin{aligned}
\hat{J_3}(0)\hat{J_2}(\gamma_{12},\gamma_{13},\gamma_{23})\ket{000}&=(\cos{\frac{\gamma_{12}}{2}}\cos{\frac{\gamma_{13}}{2}}\cos{\frac{\gamma_{23}}{2}}+i\sin{\frac{\gamma_{12}}{2}}\sin{\frac{\gamma_{13}}{2}}\sin{\frac{\gamma_{23}}{2}})f_1f_1^\dagger f_2f_2^\dagger f_3f_3^\dagger\\
&+ (\sin{\frac{\gamma_{12}}{2}}\sin{\frac{\gamma_{13}}{2}}\cos{\frac{\gamma_{23}}{2}}+i\cos{\frac{\gamma_{12}}{2}}\cos{\frac{\gamma_{13}}{2}}\sin{\frac{\gamma_{23}}{2}})f_1f_1^\dagger f_2^\dagger f_3^\dagger\\
&+(\sin{\frac{\gamma_{12}}{2}}\cos{\frac{\gamma_{13}}{2}}\sin{\frac{\gamma_{23}}{2}}+i\cos{\frac{\gamma_{12}}{2}}\sin{\frac{\gamma_{13}}{2}}\cos{\frac{\gamma_{23}}{2}})f_1^\dagger f_2f_2^\dagger f_3^\dagger\\
&+(\cos{\frac{\gamma_{12}}{2}}\sin{\frac{\gamma_{13}}{2}}\sin{\frac{\gamma_{23}}{2}}+i\sin{\frac{\gamma_{12}}{2}}\cos{\frac{\gamma_{13}}{2}}\cos{\frac{\gamma_{23}}{2}})f_1^\dagger f_2^\dagger f_3f_3^\dagger.\\
\end{aligned}
\end{equation*}

Now, the tensor product of three gates $\hat{U}_i, i=1,2,3,$ is applied to the state described above. This tensor product can be represented as
\begin{align*} \hat{U}_1(p_1)\otimes \hat{U}_2(p_2) \otimes \hat{U}_3(p_3) &=\left(\sin(p_1)(f_1-f_1^\dagger)+\cos(p_1)(f_1f_1^\dagger+f_1^\dagger f_1)\right)\\ &\otimes\left(\sin(p_2)(f_2-f_2^\dagger)+\cos(p_2)(f_2f_2^\dagger+f_2^\dagger f_2)\right)\\ &\otimes\left(\sin(p_3)(f_3-f_3^\dagger)+\cos(p_3)(f_3f_3^\dagger+f_3^\dagger f_3)\right),
\end{align*}
Whereas the serial circuit of more than two gates can be represented directly via multiplication, the sign-changing rule for the tensor product of two quantum gates cannot be simply generalized for the three gates. Thus, for the three-player cooperative game, it is necessary to establish the sign-changing rule for the general parallel circuit of three gates. When working with three gates, interactions between monomials become more complex and will have more possible effects. To handle this situation, instead of two artificial parameters, five parameters will be needed to define sign changing procedure.
\begin{enumerate}
    \item At first, we assign $a$ to monomials from the left side of multiplication with an odd number of terms of type $f_j$ and $f_j^\dagger f_j$.
    \item  Then, in the middle term of multiplication, we assign $b$ to monomials with odd number of terms of type $f_j$ and $f_j^\dagger$, $c$ to monomials, which have odd number of occurrences of terms of type $f_j^\dagger$ and $f_j$ and, at the same time, odd number of occurrences of terms of type $f_j$ and $f^\dagger_j f_j$, and $d$ to monomials, which have odd number of occurrences of terms of type $f_j$ and $f^\dagger_j f_j$.
    \item At last, we assign $e$ to monomials from the right side of multiplication with an odd number of occurrences of terms of type $f_j^\dagger$ and $f_j$. 
    \item Then, after performing the multiplication, we perform the reassignment: 
$$a,b,c,d,e,ad,be,abe,ade\rightarrow 1,$$ $$ab,ac,ae,ce,de,ace\rightarrow -1. $$ 
\end{enumerate}
Thus, before the multiplication and the reassignment, the tensor product can be represented as follows:
\begin{equation*}
    \begin{aligned}
    \hat{U}_1(p_1)\otimes \hat{U}_2(p_2) \otimes \hat{U}_3(p_3) =&\left(\sin(p_1)(af_1-f_1^\dagger)+\cos(p_1)(f_1f_1^\dagger+af_1^\dagger f_1)\right)\\
    &\left(\sin(p_2)(cf_2-bf_2^\dagger)+\cos(p_2)(f_2f_2^\dagger+df_2^\dagger f_2)\right)\\
    &\left(\sin(p_3)(ef_3-ef_3^\dagger)+\ cos(p_3)(f_3f_3^\dagger+f_3^\dagger f_3)\right).
    \end{aligned}
\end{equation*}
Alternatively, the representation of this gate in QRA can be obtained via subsequent application of the previously presented sign-changing rules for the pairs of gates. Since the considered tensor product has 64 non-zero elements, we omit its full representation and directly proceed to the case $\gamma_{123}=1$.

\subsubsection{Case $\gamma_{123}=1$}
In case $\gamma_{123}=1$, the QRA representation of the gate $$\hat{J_3}(1)=\left( \text{Id}(1)\otimes \text{CNOT}(2,3)\right) \left(\text{CNOT}(1,2)\otimes \text{Id}(3)\right) \left(\text{H}(1)\otimes  \text{Id}(2)\otimes  \text{Id}(3)\right)$$ has to be found.
The first part of the serial gate is $\left(\text{H}(1)\otimes  \text{Id}(2)\otimes  \text{Id}(3)\right)$, where
 $$\text{H}(1)=\frac{1}{\sqrt{2}}(f_1f_1^\dagger+f_1+f_1^\dagger- f_1^\dagger f_1).$$
 Since the Hadamard gate is in a tensor product with identity operators on the right side, it can be represented as a straightforward multiplication 
 \begin{align*}\text{H}(1)\otimes  \text{Id}(2)\otimes  \text{Id}(3)&=\frac{1}{\sqrt{2}}(f_1f_1^\dagger+f_1+f_1^\dagger- f_1^\dagger f_1)(f_2f_2^\dagger +f_2^\dagger f_2)(f_3f_3^\dagger +f_3^\dagger f_3).
\end{align*}
Then, the $\text{CNOT}(1,2)$ gate can be written down as 
$$\text{CNOT}(1,2)=f_1f_1^\dagger f_2f_2^\dagger+f_1f_1^\dagger f_2^\dagger f_2-f_1^\dagger f_1 f_2-f_1^\dagger f_1 f_2^\dagger.$$
 Thus, once more, we have
\begin{align*}\text{CNOT}(1,2)\otimes \text{Id}(3)&=(f_1f_1^\dagger f_2f_2^\dagger+f_1f_1^\dagger f_2^\dagger f_2-f_1^\dagger f_1 f_2-f_1^\dagger f_1 f_2^\dagger)(f_3f_3^\dagger +f_3^\dagger f_3).
\end{align*}
At last, according to the sign-changing procedure for the parallel circuit of two gates, we have 
\begin{align*}\text{Id}(1)\otimes \text{CNOT}(2,3)&=(f_1f_1^\dagger +f_1^\dagger f_1)\otimes (f_2f_2^\dagger f_3f_3^\dagger+f_2f_2^\dagger f_3^\dagger f_3-f_2^\dagger f_2 f_3-f_2^\dagger f_2 f_3^\dagger)\\
&=(f_1f_1^\dagger +af_1^\dagger f_1)(f_2f_2^\dagger f_3f_3^\dagger+f_2f_2^\dagger f_3^\dagger f_3-bf_2^\dagger f_2 f_3-bf_2^\dagger f_2 f_3^\dagger)\\
&=f_1f_1^\dagger f_2f_2^\dagger f_3f_3^\dagger+f_1f_1^\dagger f_2f_2^\dagger f_3^\dagger f_3-bf_1f_1^\dagger f_2^\dagger f_2 f_3-bf_1f_1^\dagger f_2^\dagger f_2 f_3^\dagger\\
&+af_1^\dagger f_1 f_2f_2^\dagger f_3f_3^\dagger+af_1^\dagger f_1 f_2f_2^\dagger f_3^\dagger f_3-abf_1^\dagger f_1f_2^\dagger f_2 f_3-abf_1^\dagger f_1f_2^\dagger f_2 f_3^\dagger\\
&=f_1f_1^\dagger f_2f_2^\dagger f_3f_3^\dagger+f_1f_1^\dagger f_2f_2^\dagger f_3^\dagger f_3-f_1f_1^\dagger f_2^\dagger f_2 f_3-f_1f_1^\dagger f_2^\dagger f_2 f_3^\dagger\\
&+f_1^\dagger f_1 f_2f_2^\dagger f_3f_3^\dagger+f_1^\dagger f_1 f_2f_2^\dagger f_3^\dagger f_3+f_1^\dagger f_1f_2^\dagger f_2 f_3+f_1^\dagger f_1f_2^\dagger f_2 f_3^\dagger.
\end{align*}
Thus, we are able to obtain the complete representation of the gate $\hat{J_3}(1)$ by multiplication of the above-presented terms in the corresponding order. Due to the extensive size of the gate $\hat{J_3}(1)$ and of the gates established in the previous subsections, the representation of the final state or the full parametric expression representing the quantum Shapley value can be non-informative and confusing. Therefore, we have decided to omit them. However, as it has been already demonstrated, we are able to perform quantum computing using the GAALOP. In the next section, it will be described how to measure the quantum states using GAALOP, in order to assign the resulting probabilities to the quantum Shapley value.

\section{Outcomes of the games and discussion}
At first, we calculate the resulting probabilities for the two-player quantum cooperative game using GAALOP to demonstrate the quantization code of the most simple instance. The following QRA code was compiled as a MATLAB script using GAALOPWeb.
\begin{lstlisting}[caption={Two-player cooperative game.}]
i = er1 * er2 ;
f1 = 0.5*( e1 + i * e3 );
f1T = 0.5*( e1 - i * e3 );
f2 = 0.5*( e2 + i * e4 );
f2T = 0.5*( e2 - i * e4 );
ket00 = f1 * f1T * f2 * f2T ;
J=cos(gamma/2)*(f1 * f1T * f2 * f2T+f1T * f1 * f2 * f2T+
f1 * f1T * f2T * f2+
f1T * f1 * f2T * f2)+
i*sin(gamma/2)*(-f1 * f2 -f1T * f2 + 
f1*f2T+f1T * f2T);
U1tensorU2=(sin(p1)*(a*f1-f1T)+cos(p1)*(f1*f1T+a*f1T*f1))*
(sin(p2)*(b*f2-
b*f2T)+cos(p2)*(f2*f2T+f2T*f2));
psi_final=U1tensorU2*J*ket00;
?probability_final_0=4*abs(ket00*psi_final)*
abs(ket00*psi_final);
?probability_final_2=4*abs(ket00*f2*psi_final)*
abs(ket00*f2*psi_final);
?probability_final_1=4*abs(ket00*f1*psi_final)*
abs(ket00*f1*psi_final);
?probability_final_12=4*abs(ket00*f2*f1*psi_final)*
abs(ket00*f2*f1*psi_final);
\end{lstlisting}
It provides a script describing the measurement of the resulting quantum states. After the symbolic substitution of $ab$ with $-1$ and $a,b$ with $1$, the final probabilities can be easily obtained for any given $\gamma$, $p_1$, and $p_2$. For example, for the choice $\gamma=0, p_1=3\pi/8, p_2=\pi/8$, we obtain
\begin{lstlisting}[caption={Two-player cooperative game: output of the resulting function.}]
probability_final_0 =0.1250
probability_final_1 = 0.7286
probability_final_12 =0.1250
probability_final_2 = 0.0214
\end{lstlisting}
The correctness of this result can be easily verified using \eqref{psif2}. The obtained probabilities can be straightforwardly substituted into \eqref{shapley2} to compute the quantum Shapley values of players. The previous code can be generalized for the case of a $3$-qubit system describing the three-player quantum cooperative game.

\begin{lstlisting}[caption={Three-player cooperative game.}]
i = er1 * er2 ;
f1 = 0.5*(e1 + i * e4);
f1T = 0.5*(e1 - i * e4);
f2 = 0.5*(e2 + i * e5);
f2T = 0.5*(e2 - i * e5);
f3 = 0.5*(e3 + i * e6);
f3T = 0.5*(e3 - i * e6);
ket000 = f1 * f1T * f2 * f2T*f3*f3T ;
Id1=f1*f1T+f1T*f1;
Id2=f2*f2T+f2T*f2;
Id3=f3*f3T+f3T*f3;
SWAP=f1 * f1T * f2 * f2T+f1T * f2-f1 * f2T +f1T * f1 * f2T * f2;
CNOT12=f1*f1T*f2*f2T+f1*f1T*f2T*f2-f1T*f1*f2-f1T*f1*f2T;
H1=1/sqrt(2)*(f1*f1T+f1+f1T-f1T*f1);
Id1tensorCNOT23=f1*f1T*f2*f2T*f3*f3T+f1*f1T*f2*f2T*f3T*f3-
f1*f1T*f2T*f2*f3-f1*f1T*f2T*f2*f3T+f1T*f1*f2*f2T*f3*f3T+
f1T*f1*f2*f2T*f3T*f3+f1T*f1*f2T*f2*f3+f1T*f1*f2T*f2*f3T;
J12=cos(gamma12/2)*(f1 * f1T * f2 * f2T+f1T * f1 * f2 * f2T+
f1 * f1T * f2T * f2+f1T * f1 * f2T * f2)+i*sin(gamma12/2)*
(-f1 * f2 -f1T * f2 + f1*f2T+f1T * f2T);
J13=cos(gamma13/2)*(f2 * f2T * f3 * f3T+f2T * f2 * f3 * f3T+
f2 * f2T * f3T * f3+f2T * f2 * f3T * f3)+i*sin(gamma13/2)*
(-f2 * f3 -f2T * f3 + f2*f3T+f2T * f3T);
J23=cos(gamma23/2)*(f2 * f2T * f3 * f3T+f2T * f2 * f3 * f3T+
f2 * f2T * f3T * f3+f2T * f2 * f3T * f3)+i*sin(gamma23/2)*
(-f2 * f3 -f2T * f3 + f2*f3T+f2T * f3T);
J2=SWAP*Id3*Id1*J13*SWAP*Id3*Id1*J23*J12*Id3;
J3=(1-gamma123)*Id1*Id2*Id3+gamma123*
Id1tensorCNOT23*CNOT12*Id3*H1*Id2*Id3;
U1tensorU2tensorU3=(-sin(p1)*sin(p2)*f1*f2+sin(p1)*sin(p2)*f1*f2T
+sin(p1)*cos(p2)*f1*f2T*f2+sin(p1)*cos(p2)*f1*f2*f2T-
sin(p1)*sin(p2)*f1T*f2+sin(p1)*sin(p2)*f1T*f2T
-sin(p1)*cos(p2)*f1T*f2T*f2-sin(p1)*cos(p2)*f1T*f2*f2T+
cos(p1)*sin(p2)*f1*f1T*f2-cos(p1)*sin(p2)*f1*f1T*f2T+
cos(p1)*cos(p2)*f1*f1T*f2T*f2+cos(p1)*cos(p2)*f1*f1T*f2*f2T-
cos(p1)*sin(p2)*f1T*f1*f2+cos(p1)*sin(p2)*f1T*f1*f2T
+cos(p1)*cos(p2)*f1T*f1*f2T*f2+cos(p1)*cos(p2)*f1T*f1*f2*f2T)*
cos(p3)*(f3*f3T+f3T*f3)+(-sin(p1)*sin(p2)*f1*f2+
sin(p1)*cos(p2)*f1*f2T*f2+sin(p1)*sin(p2)*f1T*f2T-
sin(p1)*cos(p2)*f1T*f2*f2T-cos(p1)*sin(p2)*f1*f1T*f2T+
cos(p1)*cos(p2)*f1*f1T*f2*f2T-cos(p1)*sin(p2)*f1T*f1*f2+
cos(p1)*cos(p2)*f1T*f1*f2T*f2)*sin(p3)*(f3-f3T)+
(-sin(p1)*sin(p2)*f1*f2T-sin(p1)*cos(p2)*f1*f2*f2T+
sin(p1)*sin(p2)*f1T*f2+sin(p1)*cos(p2)*f1T*f2T*f2-
cos(p1)*sin(p2)*f1*f1T*f2-cos(p1)*cos(p2)*f1*f1T*f2T*f2-
cos(p1)*sin(p2)*f1T*f1*f2T-cos(p1)*cos(p2)*f1T*f1*f2*f2T)*
sin(p3)*(f3-f3T);
psi_final=U1tensorU2tensorU3*J3*J2*ket000;
?probability_final_0=8*abs(ket000*psi_final)*
abs(ket000*psi_final);
?probability_final_1=8*abs(ket000*f1*psi_final)*
abs(ket000*f1*psi_final);
?probability_final_2=8*abs(ket000*f2*psi_final)*
abs(ket000*f2*psi_final);
?probability_final_3=8*abs(ket000*f3*psi_final)*
abs(ket000*f3*psi_final);
?probability_final_12=8*abs(ket000*f2*f1*psi_final)*
abs(ket000*f2*f1*psi_final);
?probability_final_13=8*abs(ket000*f3*f1*psi_final)*
abs(ket000*f3*f1*psi_final);
?probability_final_23=8*abs(ket000*f3*f2*psi_final)*
abs(ket000*f3*f2*psi_final);
?probability_final_123=8*abs(ket000*f3*f2*f1*psi_final)*
abs(ket000*f3*f2*f1*psi_final);
\end{lstlisting}

To demonstrate the functionality of our approach, we have computed the quantized version of the game from Example \ref{priklad2} under different settings. The first instance is depicted in Figure \ref{1}. 

\begin{figure}[h!]
  \centering
  \includegraphics[scale=0.1503]{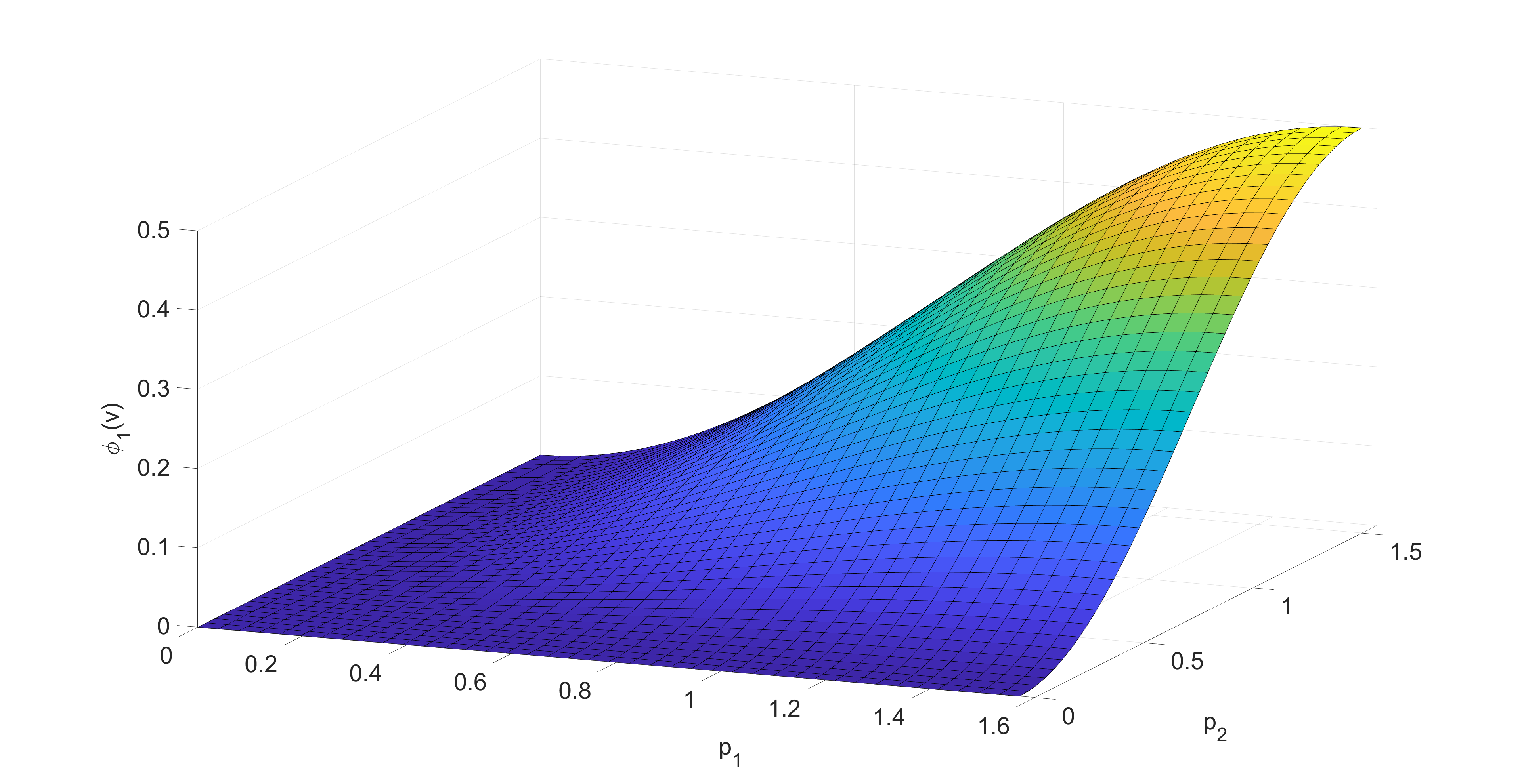}
  \caption{The quantum Shapley value $\tilde{\phi}_1$ for $\gamma_{123}=0$, $\gamma_{12}=0$, $\gamma_{13}=0$, $\gamma_{23}=0$, and $p_2=0$.}
  \label{1}
\end{figure}

When there is no bond between players and the player with the greatest weight is indifferent to a game process, player $1$ might benefit from cooperation with player $3$ ($3$ benefits as well due to the symmetric setting). The instance with the "stronger" bond between players $1$ and $3$ is depicted in Figure \ref{2}. The maximal change of the intitial state remains the best possible option for the players $1$ and $3$. However, whereas the payoff in case $p_1=p_2=\pi/2$ has decreased, players begin to benefit from not changing state at all. The instance with the maximal bond between players $1$ and $3$ is depicted in Figure \ref{3}. It can be seen that under $\gamma_{13}=\pi/2$ players' payoffs have decreased and now they maximally benefit from cooperation in a new sense: their actions have to be equivalent. Thus, both of them should completely change the initial state or not operate with it at all. The case when all players are entangled via a 3-qubit gate is depicted in Figure \ref{4}.
This setting has an analogical effect as the case depicted in Figure \ref{2}. The last case that was considered is presented in Figure \ref{5}. Compared to all previously considered instances, this last setting demonstrates that the maximal possible payoff obtained by the player does not have to correspond to boundary decisions $p_1=p_2=\pi/2$ or $p_1=p_2=0$, but can be found inside of the considered intervals. 

\begin{figure}[h!]
  \centering
  \includegraphics[scale=0.1503]{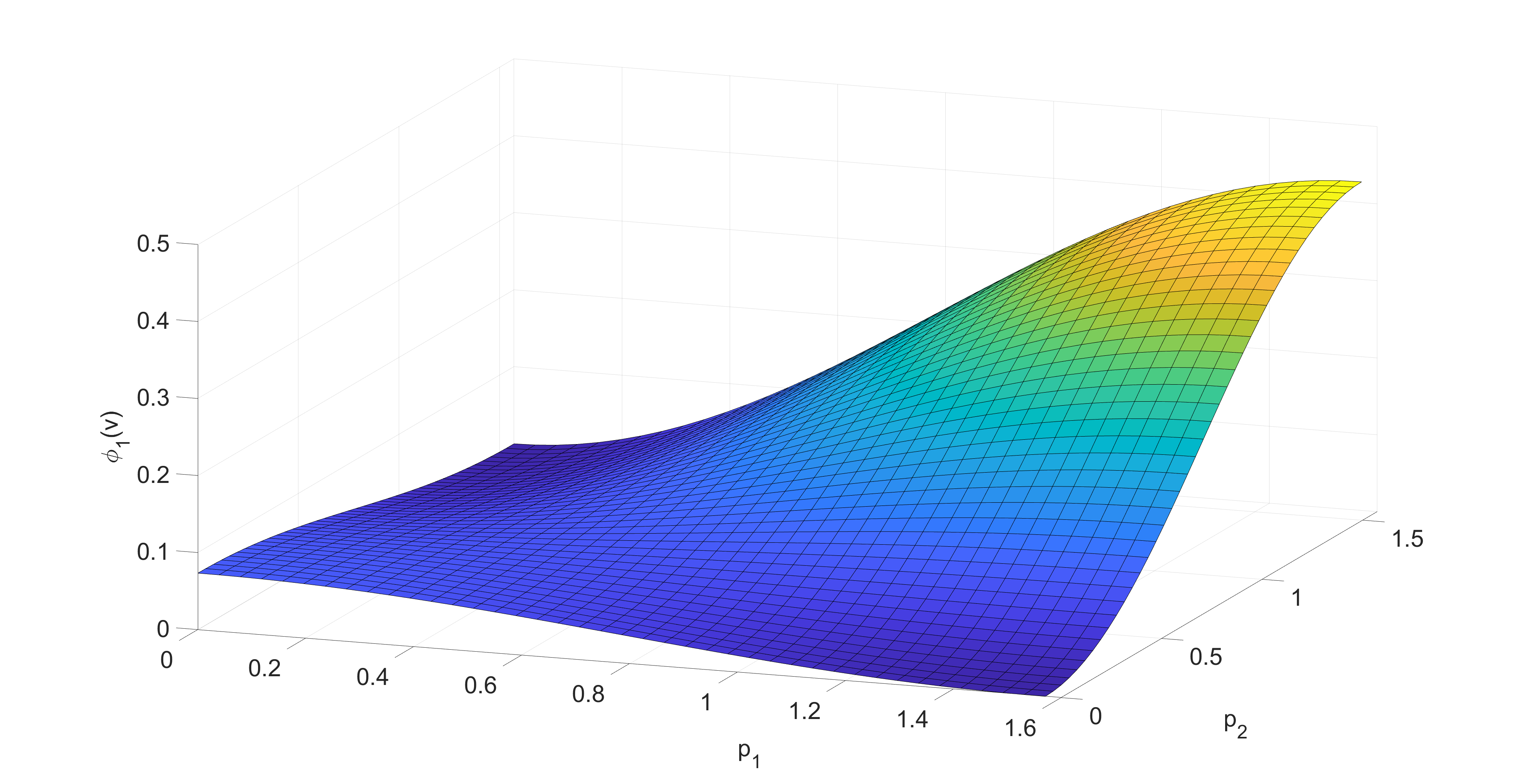}  
  \caption{The quantum Shapley value $\tilde{\phi}_1$ for $\gamma_{123}=0$, $\gamma_{12}=0$, $\gamma_{13}=\pi/4$, $\gamma_{23}=0$, and $p_2=0$.}
  \label{2}
\end{figure}

\begin{figure}[h!]
  \centering
  \includegraphics[scale=0.1503]{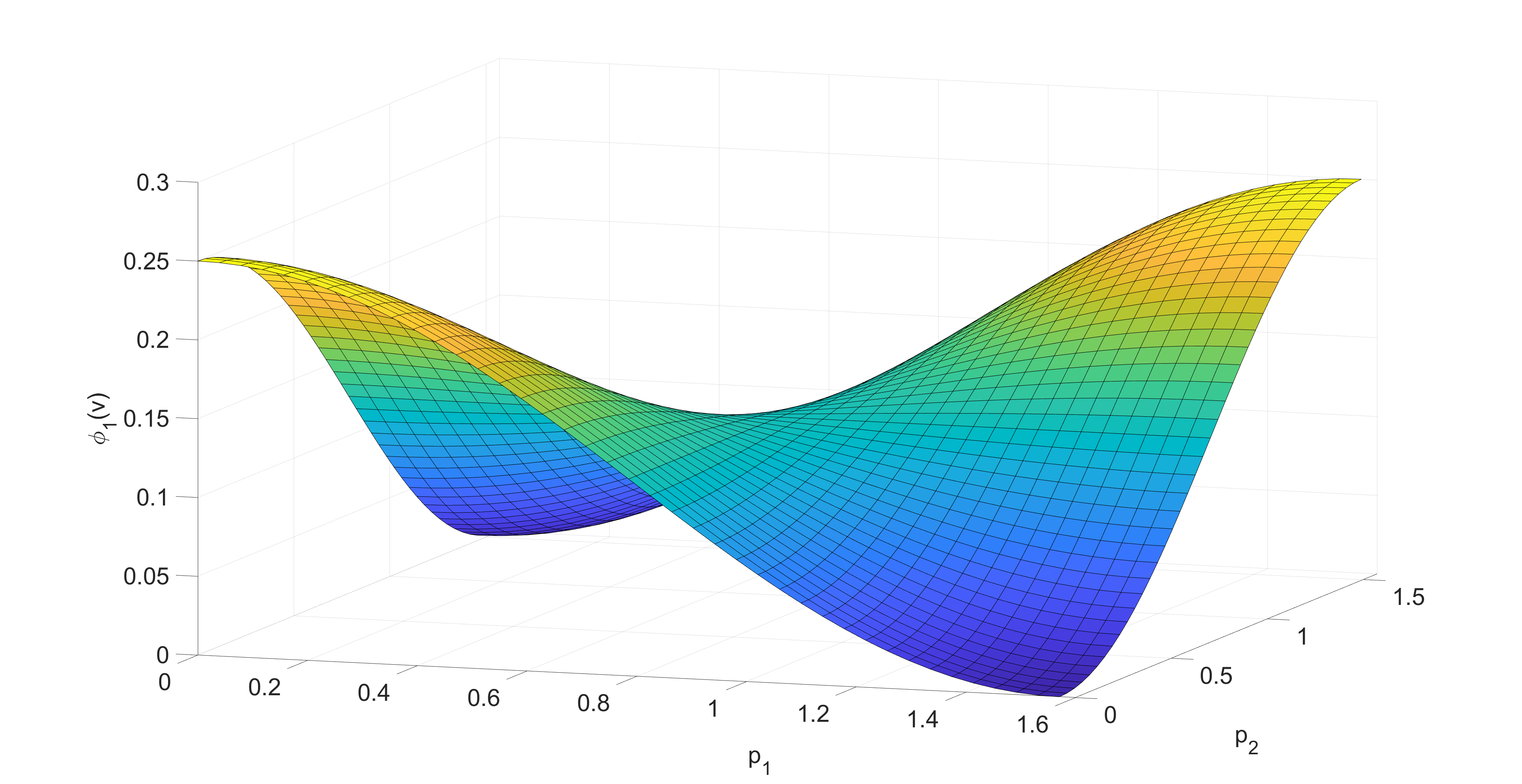} 
  \caption{The quantum Shapley value $\tilde{\phi}_1$ for $\gamma_{123}=0$, $\gamma_{12}=0$, $\gamma_{13}=\pi/2$, $\gamma_{23}=0$, and $p_2=0$.}
  \label{3}
\end{figure}

\newpage
\begin{figure}[h!]
  \centering
  \includegraphics[scale=0.1503]{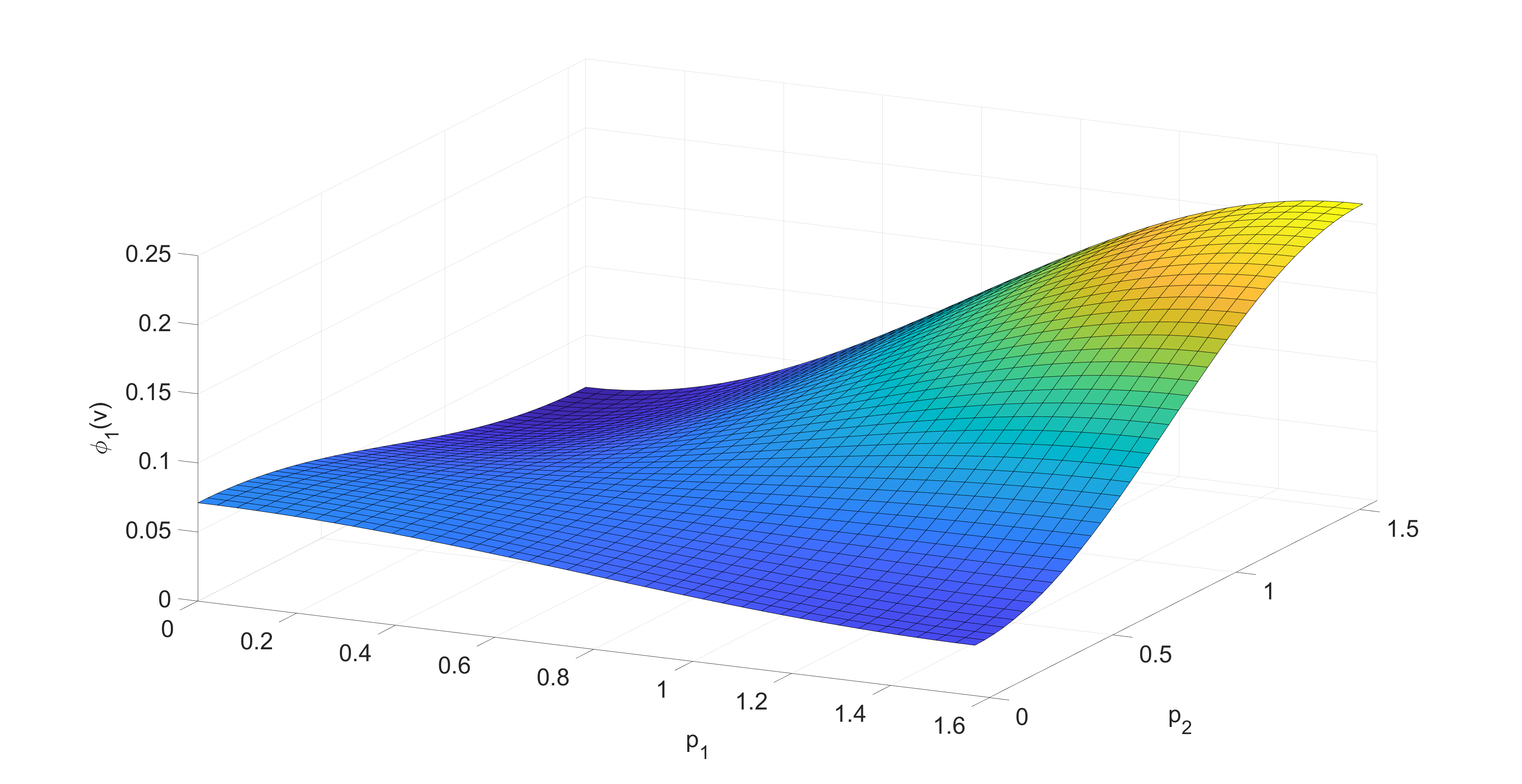}  
  \caption{The quantum Shapley value $\tilde{\phi}_1$ for $\gamma_{123}=1$, $\gamma_{12}=0$, $\gamma_{13}=\pi/4$, $\gamma_{23}=0$, and $p_2=0$.}
  \label{4}
\end{figure}

\begin{figure}[h!]
  \centering
  \includegraphics[scale=0.1503]{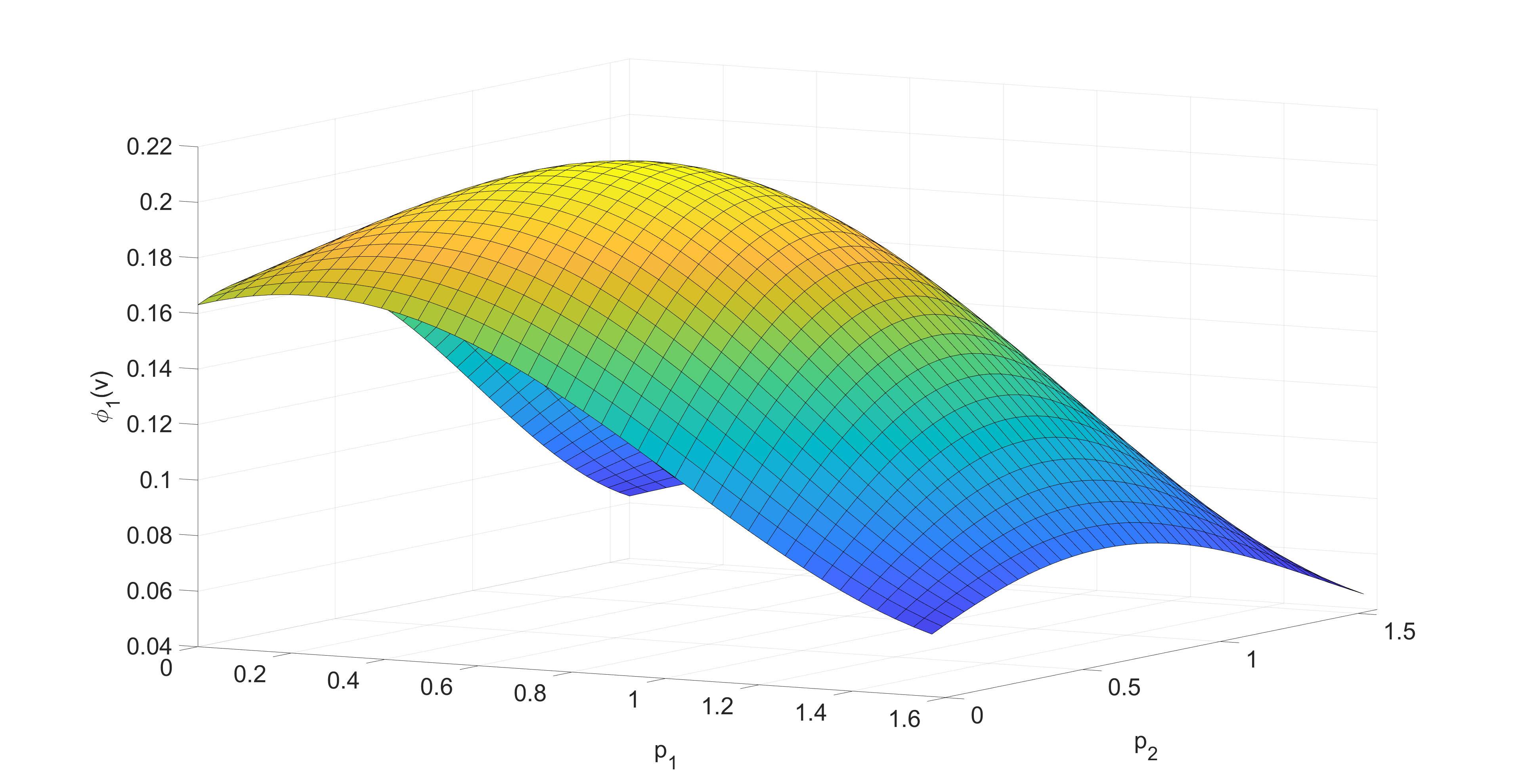} 
  \caption{The quantum Shapley value $\tilde{\phi}_1$ for $\gamma_{123}=1$, $\gamma_{12}=\pi/3$, $\gamma_{13}=3\pi/4$, $\gamma_{23}=\pi/3$, and $p_2=3\pi/4$.}
  \label{5}
\end{figure}

Figures \ref{1}-\ref{5} have demonstrated that the quantum Shapley value redistributes the payoffs according to the pre-agreement between players and takes into consideration their "acceptance" of the initial state. Indeed, players with smaller weights can benefit in situations when the player with the greatest payoff is indifferent to a game process and there is no strong bond between them. However, when equally strong players are maximally related (entangled), their distribution can only decrease, once again implying that pre-agreements only damage their prosperity. 

It can be concluded, that the proposed quantization scheme for two-player and three-player cooperative games has allowed for non-trivial results. In particular, we have demonstrated, that, depending on the properties of $v$, the strong bond between equally strong players, that are not created within the negotiation process, may decrease their payoffs (in case cooperation is beneficial). This peculiar outcome can be explained by an additional risk taken by the participants due to the existence of the pre-agreement, which can be interpreted as the initial probabilistic coalition structure. Alternatively, when cooperation in the classical cooperative games is not the best option due to the definition of $v$, then the existence of pre-agreement may improve players' payoffs. These results indicate the potential of the proposed quantization of the cooperative games. Our study has also demonstrated that QRA allows for efficient computation within the quantum game theory framework. Moreover, QRA can be even used to perform automated proofs using the geometric algebra calculator GAALOPWeb. Thus, the language of QRA and its implementation in the GAALOP have provided us with a convenient tool to perform quantum computing and to study quantum cooperative games, in particular. The future implementation of the tensor product sign-changing rule within the GAALOP shall further simplify computations with the multiple qubit states using QRA.

\end{document}